\documentclass[twoside,a4paper]{amsart}
\usepackage{amssymb,latexsym}
\input xypic
  \xyoption{all}

\usepackage{amsmath, amsthm, amsfonts}

\usepackage{pifont,pxfonts,txfonts}
\usepackage{yfonts}
\usepackage{bbm}
\usepackage{calrsfs}
\usepackage{empheq}
\usepackage{color}
\usepackage[colorlinks,linkcolor=blue,anchorcolor=blue,citecolor=blue]{hyperref}
\usepackage{amsmath, amstext, amsbsy, amscd}
\usepackage[mathscr]{eucal}
\usepackage{times}
\usepackage{amsmath, amsthm, amsfonts,extarrows}

\newtheorem{theorem}{Theorem}[section]
\newtheorem{proposition}[theorem]{Proposition}
\newtheorem{lemma}[theorem]{Lemma}
\newtheorem{remark}[theorem]{Remark}

\newtheorem{corollary}[theorem]{Corollary}

\newtheorem{conjecture}[theorem]{Conjecture}

   \newcommand{\ba}{\begin{eqnarray}}
   \newcommand{\na}{\end{eqnarray}}
   \newcommand{\ban}{\begin{eqnarray*}}
   \newcommand{\nan}{\end{eqnarray*}}

\newcommand{\bC}{{\mathbb C}}

\newcommand{\bP}{{\mathbb P}}
\newcommand{\bQ}{{\mathbb Q}}

\newcommand{\bZ}{{\mathbb Z}}

\newcommand{\sC}{{\mathscr C}}
\newcommand{\sX}{{\mathscr X}}

\newcommand{\tspec}{\textrm{Spec}}

  \newcommand{\<}{\langle}
  \renewcommand{\>}{\rangle}

\newcommand{\suml}{\sum\limits}
\newcommand{\prodl}{\prod\limits}

\begin{document}

\title{On semisimplicity of quantum cohomology of $\bP^1$-orbifolds}

\author[Hua-Zhong Ke]{Hua-Zhong Ke}
\address{School of Mathematics, Sun Yat-sen University, Guangzhou 510275, P.R. China;}
\email{kehuazh@mail.sysu.edu.cn}

\maketitle

\begin{abstract}
For a $\bP^1$-orbifold $\sC$, we prove that its big quantum cohomology is generically semisimple. As a corollary, we verify a conjecture of Dubrovin for orbi-curves. We also show that the small quantum cohomology of $\sC$ is generically semisimple iff $\sC$ is Fano, i.e. it has positive orbifold Euler characteristic.
\end{abstract}





\section{Introduction}

Quantum cohomology stems from genus-zero Gromov-Witten theory, which concerns virtual counts of rational curves in target manifolds or orbifolds. One can naively view the quantum cohomology ring as a deformation of the ordinary cohomology ring. A fundamental problem in Gromov-Witten theory is to understand the topology and geometry of target spaces hidden behind the algebraic structure of their quantum cohomology. 

Unlike the ordinary cohomology, the quantum cohomology can be semisimple for some targets, and deep structural results for semisimple Gromov-Witten theories are known (e.g. Givental-Teleman's reconstruction theorem \cite{G,T}). It is important to understand the geometry of such semisimple target spaces. One of the most important conjectures in this direction was proposed by Dubrovin \cite{D} in his ICM talk in 1998 (later made more precise in \cite{B04,HMT}):
\begin{conjecture}\label{conj}
For a smooth projective variety $X$, the followings are equivalent:
\begin{enumerate}
\item The (even parity) big quantum cohomology $QH(X)$ is generically semisimple.
\item The bounded derived category of coherent sheaves $D^b(X)$ admits a full exceptional collection.
\end{enumerate}
\end{conjecture}

In the last two decades, only a few examples of smooth projective varieties with semisimple quantum cohomology are known, since it is difficult to check the semisimplicity. Such examples include projective toric manifolds \cite{I}, certain rational homogeneous spaces and hyperserfaces inside them \cite{CMP,CP,P}, rational surfaces \cite{B04,BM}, certian Fano threefolds \cite{C}, blow-up of $\bP^3$ along a smooth rational curve \cite{Ke}, and blow-ups of such varieties at points \cite{B04}. To the knowledge of the author, all known semisimple examples admit full exceptional collections.

It is natural to generalize Dubrovin's conjecture to orbifolds. In this article, we will prove this conjecture for orbi-curves. An orbi-curve $\sC$ is a complex orbifold with trivial generic stablizer, whose underlying space $|\sC|$ is a compact Riemann surface (in the literature, an orbi-curve is also called an orbifold Riemann surface). If $|\sC|\cong\bP^1$, then we say that $\sC$ is a $\bP^1$-orbifold.

It was shown by Geigle-Lenzing in the 80's that $\bP^1$-orbifolds admit full exceptional collections \cite{GL}. This inspires us to prove the following proposition.

\begin{proposition}\label{big}
Let $\sC$ be a $\bP^1$-orbifold. Then $QH(\sC)$ is generically semisimple.
\end{proposition}

As a corollary, we verify Dubrovin's conjecture for orbi-curves.

\begin{corollary}\label{cor}
Let $\sC$ be an orbi-curve. Then $QH(\sC)$ is generically semisimple iff $D^b(\sC)$ admits a full exceptional collection.
\end{corollary}

It was observed that, to formulate Dubrovin's conjecture for smooth projective varieties, we need $QH(X)$ instead of $qH(X)$, since there are smooth projective varieties $X$ such that $QH(X)$ is generically semisimple while $qH(X)$ is not. Here $qH(X)$ is the (even parity) small quantum cohomology of $X$. The first known example of this kind is $IG(2,6)$ \cite{GMS}, and up until now, only $IG(2,2n)(n\geq3)$ and $F_4/P_4$ are proved to have this pattern (see Theorem 4 in \cite{P}). Note that all these examples have dimensions at least seven. In the category of orbifolds, our second main result shows that such phenomena appear in dimension one, and non-Fano $\bP^1$-orbifolds give a new class of examples.

\begin{proposition}\label{small}
Let $\sC$ be an orbi-curve. Then $qH(\sC)$ is generically semisimple iff $\sC$ is Fano, i.e. it has positive orbifold Euler characteristic.
\end{proposition}

Recall that the orbifold Euler characteristic of an orbi-curve $\sC$ is
\ban
\chi_{orb}(\sC)=\chi_{top}(|\sC|)-\suml_{p\in|\sC|}(1-\frac{1}{a_p}),
\nan
where $a_p$ is the order of $p$. Note that $a_p$ is larger than $1$ for only finitely many $p$. We say that $\sC$ is Fano if $\chi_{orb}(\sC)>0$, $\sC$ is Calabi-Yau if $\chi_{orb}(\sC)=0$, and $\sC$ is of general type if $\chi_{orb}(\sC)<0$. One can check that $\sC$ is Fano iff it is one of the followings:
\ban
\bP^1_{a_1,a_2}(a_1,a_2\geq1),\bP^1_{2,2,a}(a\geq2),\bP^1_{2,3,a}(a=3,4,5),
\nan
and $\sC$ is Calabi-Yau iff it is one the followings:
\ban
\textrm{elliptic curves}, \bP^1_{2,2,2,2}, \bP^1_{3,3,3}, \bP^1_{2,4,4}, \bP^1_{2,3,6}.
\nan
Here for a Riemann surface $C$ and a tuple of positive integers $\mathbf a=(a_1,\cdots,a_r)$, we use $C_\mathbf a$ to denote an orbi-curve with underlying space $C$ and $r$ distinct (possibly trivial) orbifold points with local groups $\mu_{a_1},\cdots,\mu_{a_r}$. It is well-known that every orbi-curve has this form.

To prove Proposition \ref{big}, our strategy is to show the invertibility of $e_q$, the quantum Euler class introduced by Abrams in \cite{A}, since the semisimplicity of the big quantum cohomology is equivalent to the invertibility of $e_q$ (Theorem 3.4 in \cite{A}). For a $\bP^1$-orbifold, we will show that the degree-zero part of $\det(e_q\star)$ vanishes, but the degree-one part is non-vanishing. Besides the dimension axiom and WDVV, a key ingredient in the computation of $\det(e_q\star)$ is a decomposition result of the degree-zero and degree-one parts of the genus-zero (primary) potential of orbi-curves (Proposition \ref{decompositionAB}). This decomposition comes from the computation of the genus-zero potential by Rossi \cite{R}, who used Symplectic Field Theory (SFT) technique \cite{EGH} to express the potential in terms of connected Hurwitz numbers and SFT invariants of orbifold caps. For Proposition \ref{small}, we use the dimension axiom to prove the ``only if" part, and to prove the ``if" part, we use the explicit presentation of the small quantum cohomology of Fano orbi-curves, which was obtained or implicitly known in \cite{MT, R, IST12}.

We remark that the semisimplicity of $QH(\sC)$ was known only for $\sC=\bP^1_{a_1,a_2,a_3}$ and $\bP^1_{2,2,2,2}$, and the semisimplicity of $qH(\sC)$ was known only for $\sC=\bP^1_{a_1,a_2}$ \cite{MT, R, KS, ST}. The existing methods in the literature do not seem to work for general cases.

In this article, we only consider orbi-curves which are effective in the sense that the generic stablizer is trivial. For an ineffective orbi-curve $\sC$, we conjecture that its big quantum cohomology is  generically semisimple iff its underlying space is $\bP^1$, and its small quantum cohomology is generically semisimple iff its rigidification \cite{ACV,BN} is an effective Fano orbi-curve. We hope to study this in the future.

The rest of the article is organized as follows. In Section 2, we briefly review some basic materials on quantum cohomology of orbi-curves, and give several different but equivalent characterization of semisimplicity. In Section 3, we prove Proposition \ref{big} and Corollary \ref{cor}. In Section 4, we prove Proposition \ref{small}.

\section{Preliminaries}

In this section, we briefly review some basic materials on quantum cohomology of orbi-curves and fix notations used throughout the rest of the article. We also give several different but equivalent characterization of semisimplicity of quantum cohomology. 

\subsection{Quantum cohomology of orbi-curves}

In this subsection, we assume the readers have some familiarity with orbifold quantum cohomology, and we refer interested readers to \cite{AGV,CR02,CR04} for details.

We have the following decomposition of the inertia orbifold of $C_\mathbf a$ into disjoint union of connected components:
\ban
IC_\mathbf a=C_\mathbf a\sqcup\bigsqcup_{\alpha=1}^r\bigsqcup_{i=1}^{a_\alpha-1}B\mu_{a_\alpha}(i).
\nan
Here $B\mu_{a_\alpha}(i)\cong B\mu_{a_\alpha}$, which is the classifying stack of the group of $a_\alpha$-th roots of units. Then the (even parity) orbifold cohomology group of $C_\mathbf a$ is
\ban
H^{even}_{orb}(C_\mathbf a)=H^{even}(IC_\mathbf a)=H^0(C_\mathbf a)\oplus H^2(C_\mathbf a)\oplus\bigoplus_{\alpha=1}^r\bigoplus_{i=1}^{a_\alpha-1}H^0(B\mu_{a_\alpha}(i)).
\nan
Here ``even" means we only consider classes of even topological degree.

Fix an index set
\ban
\mathscr S=\{(0,0),(0,1)\}\sqcup\mathscr T\textrm{ with }\mathscr T=\bigsqcup_{\alpha=1}^r\{(\alpha,1),\cdots,(\alpha,a_\alpha-1)\},
\nan
and set
\ban
\phi_{00}&:=&\mathbbm1\in H^0(C_\mathbf a),\\
\phi_{01}&:=&[point]\in H^2(C_\mathbf a),\\
\phi_{\alpha i}&:=&\mathbbm1\in H^0(B\mu_{a_\alpha}(i)).
\nan
Then $\mathscr B=\{\phi_s\}_{s\in\mathscr S}$ is a basis of $H^{even}_{orb}(C_\mathbf a)$, which is homogeneous with respect to the orbifold degree. Here the orbifold degrees of $\phi_{00}$ and $\phi_{01}$ are their topological degrees, and the orbifold degree of $\phi_{\alpha,i}$ is $\frac{2i}{a_\alpha}$.

In terms of classes in $\mathscr B$, the orbifold Poincar\'e pairing of $C_\mathbf a$ is given by
\ban
\<\phi_{00},\phi_{01}\>_{orb}^{C_\mathbf a}=1,\quad\<\phi_{\alpha i},\phi_{\alpha,a_\alpha-i}\>_{orb}^{C_\mathbf a}=\frac{1}{a_\alpha},\textrm{ and }0\textrm{ otherwise}.
\nan
Let $g_{s's''}=\<\phi_{s'},\phi_{s''}\>_{orb}^{C_\mathbf a}$. Then the matrix $(g_{s's''})$ is nonsingular, and we let $(g^{s's''})=(g_{s's''})^{-1}$. Let $\{\phi^s\}_{s\in\mathscr S}$ be the dual basis of $\mathscr B$ with respect to the orbifold Poincar\'e pairing. Then
\ban
\phi^{00}=\phi_{01},\quad\phi^{01}=\phi_{00},\quad\phi^{\alpha,i}=a_\alpha\phi_{\alpha,a_\alpha-i}.
\nan

The Chen-Ruan product $\cup_{CR}$ on $H^{even}_{orb}(C_\mathbf a)$ satisfies
\ba\label{CR1}
\underbrace{\phi_{\alpha1}\cup_{CR}\cdots\cup_{CR}\phi_{\alpha1}}_k=\left\{\begin{array}{cc}
\phi_{\alpha k},&1\leq k\leq a_\alpha-1,\\\\
\frac 1a_\alpha\phi_{01},&k=a_\alpha,\\\\
0,&k\geq a_\alpha+1,
\end{array}\right.
\na
and 
\ba\label{CR2}
\phi_{\alpha_1,1}\cup_{CR}\phi_{\alpha_2,1}=0(\alpha_1\neq\alpha_2).
\na
We remark that $\cup_{CR}$ also respects the orbifold degree.

The genus-zero potential of $C_\mathbf a$ is a formal function of $\mathbf t=\suml_{s\in\mathscr S}t^s\phi_s\in H^*_{orb}(C_\mathbf a)$ given by
\ban
F(\mathbf t)=\suml_{d=0}^\infty\suml_{m=0}^\infty\<\underbrace{\mathbf t,\cdots,\mathbf t}_m\>^{C_\mathbf a}_{0,m,d}\frac{Q^d}{m!}=\suml_{d=0}^\infty\suml_{m_s\geq0,s\in\mathscr S}\<\bigotimes_{s\in\mathscr S}\phi_s^{\otimes m_s}\>_{0,\suml_{s\in\mathscr S}m_s,d}^{C_\mathbf a}\prodl_{s\in\mathscr S}\frac{(t^s)^{m_s}}{m_s!}Q^d,
\nan
where $\<\bigotimes_{s\in\mathscr S}\phi_s^{\otimes m_s}\>_{0,\suml_{s\in\mathscr S}m_s,d}^{C_\mathbf a}$ is a Gromov-Witten invariant of $C_\mathbf a$ of genus-zero, degree-$d$.
Write
\ban
\mathbf t=t^{00}\phi_{00}+t^{01}\phi_{01}+\mathbf t^\mathbf a,\textrm{ with }\mathbf t^\mathbf a=\suml_{s\in\mathscr T}t^s\phi_s.
\nan
Then from properties of Gromov-Witten invariants (dimension axiom, fundamental class axiom, divisor axiom), the potential has the following form:
\ban
F=\frac 12(t^{00})^2t^{01}+\suml_{\alpha=1}^r\suml_{i=1}^{a_\alpha-1}\frac{t^{00}t^{\alpha,i}t^{\alpha,{a_\alpha-i}}}{2a_\alpha}+A^\mathbf a+\suml_{d=1}^\infty B^\mathbf a_d(Qe^{t^{01}})^d,
\nan
where
\ban
A^\mathbf a&=&\suml_{m=3}^\infty\frac1{m!}\<\underbrace{\mathbf t^\mathbf a,\cdots,\mathbf t^\mathbf a}_m\>_{0,m,0}^{C_\mathbf a}\in\bQ[\{t^s\}_{s\in\mathscr T}],\\
B^\mathbf a_d&=&\suml_{m=0}^\infty\frac1{m!}\<\underbrace{\mathbf t^\mathbf a,\cdots,\mathbf t^\mathbf a}_m\>^{C_\mathbf a}_{0,m,d}\in\bQ[\{t^s\}_{s\in\mathscr T}].
\nan
Set
\ba\label{degt}
\deg t^{00}=1,\deg t^{01}=0,\deg t^{\alpha,i}=1-\frac ia_\alpha,\deg Qe^{t^{01}}=\chi_{orb}(C_\mathbf a).
\na
Then we can use the dimension axiom to show that $F,A^\mathbf a,B^\mathbf a_d$ are weighted homogeneous with degree
\ba\label{degf}
\deg F=\deg A^\mathbf a=2,\quad\deg B^\mathbf a_d=2-d\cdot\chi_{orb}(C_\mathbf a).
\na

The big quantum product is given by
\ban
\phi_{s_1}\star_\mathbf t\phi_{s_2}=\suml_{s\in\mathscr S}F_{s_1,s_2,s}(\mathbf t)\phi^s,
\nan
with coefficients in $\bC[\{t^s\}_{s\in\mathscr T}][[Qe^{t^{01}}]]$. So the big quantum cohomology of $C_\mathbf a$ is
\ban
QH(C_\mathbf a)=H^{even}_{orb}(C_\mathbf a)\otimes_\bC\bC[\{t^s\}_{s\in\mathscr T}][[Qe^{t^{01}}]].
\nan
The big quantum product is clearly commutative, but it is a highly nontrivial fact that it is associative, which is due to the famous WDVV equations satisfied by the genus-zero potential. For $s_1,s_2,s_3,s_4\in\mathscr S$, the WDVV of type $(s_1,s_2;s_3,s_4)$ reads
\ba\label{WDVV}
\suml_{s',s''\in\mathscr S}F_{s_1,s_2,s'}g^{s's''}F_{s'',s_3,s_4}=\suml_{s',s''\in\mathscr S}F_{s_1,s_3,s'}g^{s's''}F_{s'',s_2,s_4}.
\na

The small quantum product is given by
\ban
\phi_{s_1}\circ\phi_{s_2}=\phi_{s_1}\star_{\mathbf t}\phi_{s_2}|_{\mathbf t=0}=\phi_{s_1}\cup_{CR}\phi_{s_2}+\suml_{s\in\mathscr S}\suml_{d=1}^\infty\<\phi_{s_1},\phi_{s_2},\phi_s\>_{0,3,d}^{C_\mathbf a}Q^d\phi^s,
\nan
with coefficients in $\bC[[Q]]$. So the small quantum cohomology of $C_\mathbf a$ is
\ban
qH(C_\mathbf a)=H^{even}_{orb}(C_\mathbf a)\otimes_\bC\bC[[Q]].
\nan
Since $qH(C_\mathbf a)$ is obtained from $QH(C_\mathbf a)$ by modding out $\mathbf t$, it follows that $qH(C_\mathbf a)$ is also commutative and associative.

\subsection{Semisimplicity of quantum cohomology}

All rings and algebras in this subsection are commutative.

Let $k$ be a field of characteristic zero, and let $R$ be a finite dimensional $k$-algebra. We say that $R$ is semisimiple over $k$ if $R$ contains no nilpotent elements, i.e. $\tspec R$ is reduced. Equivalently, $R$ is isomorphic to a product $\prodl_{i=1}^mk_i$, where each $k_i$ is a finite extension field of $k$.

Let $A$ be a $k$-algebra (not necessarily finite dimensional) which is also an integral domain, and let $B$ be an $A$-algebra which is freely finitely generated as an $A$-module. We say that $B$ is (generically) semisimple over $A$ if there exists a non-empty open subset $U$ of $\tspec(A)$ such that, for every $\mathfrak p\in U$, $(\tspec(B))_\mathfrak p$ is reduced, i.e., $B\otimes_Ak(\mathfrak p)$ is semisimple over $k(\mathfrak p)$.

The following lemma is well-known.
\begin{lemma}\label{ABsemisimple}
The followings are equivalent.
\begin{enumerate}
\item $B$ is generically semisimple over $A$.
\item There exists $\mathfrak p\in\tspec(A)$ such that $(\tspec(B))_\mathfrak p$ is semisimple over $k(\mathfrak p)$.
\item $(\tspec(B))_\eta$ is semisimple over $k(\eta)$, where $\eta=(0)\in\tspec(A)$ is the generic point.
\end{enumerate}
\end{lemma}

A direct corollary of the above lemma is the following.
\begin{corollary}\label{nilpotent}
$B$ is generically semisimple over $A$ iff $B$ contains no nilpotent elements.
\end{corollary}

For quantum cohomology, we choose $k=\bC$, $A_Q=\bC[\{t^s\}_{s\in\mathscr T}][[Qe^{t^{01}}]]$, $B_Q=QH(C_\mathbf a)$, $A_q=\bC[[Q]]$ and $B_q=qH(C_\mathbf a)$. Then we have the following Cartesian diagram:
\[
\begin{CD}
\tspec(B_q)@>>>\tspec(B_Q)\\
@VVV @VVV \\
\tspec(A_q)@>>> \tspec(A_Q),
\end{CD}
\]
where the base morphism is the inclusion given by modding out $\mathbf t$ from $A_Q$. So the semisimiplicity of $qH(C_\mathbf a)$ implies that of $QH(C_\mathbf a)$, but the converse is not true (Proposition \ref{small}).

We will use Lemma \ref{ABsemisimple} and Corollary \ref{nilpotent} to deal with the semisimiplicity of the small quantum cohomology. For the big case, we need another ingredient. As in (3) of Lemma \ref{ABsemisimple}, let $\eta$ be the generic point of $A_Q$. Then $\tilde B_Q=B_Q\otimes_{A_Q}k(\eta)$ is a finite-dimensional Frobenius algebra over $k(\eta)$, and it has a distinguished element called quantum Euler class, introduced by Abrams \cite{A}:
\ba\label{qeuler}
e_q=\suml_{s\in\mathscr S}\phi_s\star_\mathbf t\phi^s\in B_Q\subset\tilde B_Q.
\na
\begin{lemma}\label{qeulerinvertible}
$\tilde B_Q$ is semisimple over $k(\eta)$ iff $\det(e_q\star_\mathbf t)\neq0$.
\end{lemma}
\begin{proof}
This is a special case of Theorem 3.4 in \cite{A}.
\end{proof}

We will use Lemma \ref{ABsemisimple} and \ref{qeulerinvertible} to deal with the semisimplicity of the big quantum cohomology.

\if{\section{Big for order $2$}

Let $\sX_r=\bP^1_{\underbrace{2,\cdots,2}_r}$ for $r\geq0$.

\begin{proposition}
The big quantum cohomology of $\sX_r$ is semisimple.
\end{proposition}

The genus-zero potential of $\sX_r$ is
\ban
F=\frac 12t_0^2s+\frac 14t_0\suml_{i=1}^rt_i^2-\frac{1}{96}\suml_{i=1}^rt_i^4+e^s\prodl_{i=1}^rt_i+O(e^{2s}).
\nan
So the quantum product is given by
\ban
x_i\star x_i&=&-\frac{t_i}{2}x_i+\frac p2+O(e^{2s}),\\
x_i\star x_j&=&e^s\frac{\prodl_{i'=1}^rt_{i'}}{t_it_j}+\suml_{k\neq i,j}2e^s\frac{\prodl_{i'=1}^rt_{i'}}{t_it_jt_k}x_k+O(e^{2s}),(i\neq j)\\
x_i\star p&=&e^s\frac{\prodl_{i'=1}^rt_{i'}}{t_i}+\suml_{k\neq i}2e^s\frac{\prodl_{i'=1}^rt_{i'}}{t_it_k}x_k+O(e^{2s}),\\
p\star p&=&e^s\prodl_{i'=1}^rt_{i'}+\suml_{k=1}^r2e^s\frac{\prodl_{i'=1}^rt_{i'}}{t_k}x_k+O(e^{2s}).
\nan
Then the quantum Euler class is
\ban
e_q=\mathbbm 1\star p+\suml_{i=1}^rx_i\star 2x_i+p\star\mathbbm1=-\suml_{i=1}^rt_ix_i+(r+2)p+O(e^{2s}),
\nan
and the matrix of $(e_q\star)$ with respect to the basis $\{\mathbbm1,x_1,\cdots,x_r,p\}$ is
\ban
M=\left[\begin{array}{cccccc}
0&3e^s\frac{\prodl_{i'=1}^rt_{i'}}{t_1}&3e^s\frac{\prodl_{i'=1}^rt_{i'}}{t_2}&\cdots&3e^s\frac{\prodl_{i'=1}^rt_{i'}}{t_r}&2e^s\prodl_{i'=1}^rt_{i'}\\
-t_1&\frac{t_1^2}{2}&8e^s\frac{\prodl_{i'=1}^rt_{i'}}{t_2t_1}&\cdots&8e^s\frac{\prodl_{i'=1}^rt_{i'}}{t_rt_1}&6e^s\frac{\prodl_{i'=1}^rt_{i'}}{t_1}\\
-t_2&8e^s\frac{\prodl_{i'=1}^rt_{i'}}{t_1t_2}&\frac{t_2^2}{2}&\cdots&8e^s\frac{\prodl_{i'=1}^rt_{i'}}{t_rt_2}&6e^s\frac{\prodl_{i'=1}^rt_{i'}}{t_2}\\
\vdots&\vdots&\vdots&\ddots&\vdots&\vdots\\
-t_r&8e^s\frac{\prodl_{i'=1}^rt_{i'}}{t_1t_r}&8e^s\frac{\prodl_{i'=1}^rt_{i'}}{t_2t_r}&\cdots&\frac{t_r^2}{2}&6e^s\frac{\prodl_{i'=1}^rt_{i'}}{t_r}\\
r+2&-\frac{t_1}{2}&-\frac{t_2}{2}&\cdots&-\frac{t_r}{2}&0
\end{array}\right]+O(e^{2s}).
\nan
One can check that 
\ban
\det M=-e^s\frac{\prodl_{i'=1}^rt_{i'}^3}{2^{r-2}}+O(e^{2s}),
\nan
which implies that $e_q$ is invertible in the big quantum cohomology algebra of $\sX_r$. As a consequence, the big quantum cohomology algebra of $\sX_r$ is semisimple.}\fi

\section{Big quantum cohomology}

In this section, we prove Proposition \ref{big} and Corollary \ref{cor}.

\subsection{Genus-zero potential}

Let $\mathbf a=(a_1,\cdots,a_r)$ be a tuple of positive integers, with $r\geq1$ and each $a_\alpha\geq2$. From Lemma \ref{ABsemisimple} and \ref{qeulerinvertible}, to prove the semisimplicty of $QH(\bP^1_\mathbf a)$, it suffices to show that $\det(e_q\star_\mathbf t)\neq0$. Note that $\det(e_q\star_\mathbf t)\in\bC[\{t^s\}_{s\in\mathscr T}][[Qe^{t^{01}}]]$, and we will show that the coefficient of $\det(e_q\star_\mathbf t)$ at $(Qe^{t^{01}})^0$ is zero, but the coefficient at $(Qe^{t^{01}})^1$ is nonvanishing. To this end, we need to understand the structure of the degree-zero and degree-one parts of the genus-zero potential of $\bP^1_\mathbf a$. 

Recall that the genus-zero potential for $\bP^1_\mathbf a$ is
\ba\label{potential}
F=\frac 12(t^{00})^2t^{01}+\suml_{\alpha=1}^r\suml_{i=1}^{a_\alpha-1}\frac{t^{00}t^{\alpha,i}t^{\alpha,{a_\alpha-i}}}{2a_\alpha}+A^\mathbf a+\suml_{d=1}^\infty B^\mathbf a_d(Qe^{t^{01}})^d.
\na
Rossi \cite{R} used the Symplectic Field Theory (SFT) technique \cite{EGH} to express the potential in terms of SFT invariants of orbifold caps and connected Hurwitz numbers. The result is  (see formula (2) in \cite{R})
\ba
\suml_{\alpha=1}^r\suml_{i=1}^{a_\alpha-1}\frac{t^{00}t^{\alpha,i}t^{\alpha,{a_\alpha-i}}}{2a_\alpha}+A^\mathbf a(\mathbf t^\mathbf a)=\suml_{\alpha=1}^r\mathbf F_{a_\alpha;0}(t^{00},t^{\alpha,1},\cdots,t^{\alpha,a_\alpha-1}),\label{potentialA}
\na
and
\ba
B^\mathbf a_d(\mathbf t^\mathbf a)=\suml_{|\mu^1|,\cdots,|\mu^r|=d}H^0_{0,d}(\mu^1,\cdots,\mu^r)\prodl_{\alpha=1}^r\prodl_{w=1}^{a_\alpha}\mathbf F_{a_\alpha;w}(t^{\alpha1},\cdots,t^{\alpha,a_\alpha-1})^{m^\alpha_w}.\label{potentialB}
\na
Here each $\mu^\alpha=(1^{m^\alpha_1}2^{m^\alpha_2}\cdots)$ is a partition of $d$ with $m^\alpha_w=0$ for $w>a_\alpha$, and $H^0_{0,d}(\mu^1,\cdots,\mu^r)$ is the Hurwitz number of genus-zero, degree-$d$ connected coverings over $\bP^1$ with ramification profile $\mu^1,\cdots,\mu^r$. Moreover, $\mathbf F_{a;w}(0\leq w\leq a)$ come from the SFT potential of the orbifold cap $[\bC/\mu_a]$: 
\ban
\mathbf F_a=\frac 1\hbar(\mathbf F_{a;0}+\suml_{w=1}^a\mathbf F_{a;w}\frac{p_w}{w}),
\nan 
with
\ban
\mathbf F_{a;w}=\left\{\begin{array}{cc}
\suml_{\substack{j_0,j_1,\cdots,j_{a-1}\in\bZ_{\geq0}\\\suml_{k=0}^{a-1}(a-k)j_k=2a}}A^a_{j_0,j_1,\cdots,j_{a-1}}\prodl_{i=0}^{a-1}(t^i)^{j_i},&w=0,\\
\suml_{\substack{j_1,\cdots,j_{a-1}\in\bZ_{\geq0}\\\suml_{k=1}^{a-1}(a-k)j_k=a-w}}B^a_{j_1,\cdots,j_{a-1};w}\prodl_{i=1}^{a-1}(t^i)^{j_i},&1\leq w\leq a.
\end{array}\right.
\nan
Here $A^a_{j_0,j_1,\cdots,j_{a-1}}$'s and $B^a_{j_1,\cdots,j_{a-1};w}$'s are SFT invariants of the orbifold cap $[\bC/\mu_a]$. We refer interested readers to \cite{EGH, R} for detailed explanations of SFT techniques.

The following result follows directly from \eqref{potential}, \eqref{potentialA}, \eqref{potentialB}, which will be used in Section 3.3.
\begin{proposition}\label{decompositionAB}
We have the following decomposition for the degree-zero and degree-one parts of the genus-zero potential of $\bP^1_\mathbf a$:
\ban
A^\mathbf a(\mathbf t^\mathbf a)=\suml_{\alpha=1}^rA^{a_\alpha}(t^{\alpha,1},\cdots,t^{\alpha,a_\alpha-1}),\quad B^\mathbf a_1(\mathbf t^\mathbf a)=\prodl_{\alpha=1}^rB^{a_\alpha}_1(t^{\alpha,1},\cdots,t^{\alpha,a_\alpha-1}).
\nan
\end{proposition}

\subsection{Special case: tear drops}

In this subsection, as a warmup, we prove the semisimplicity of big quantum cohomology of tear drops $\bP^1_a(a\geq2)$. To ease notations, throughout this subsection, for $1\leq i\leq a-1$, we set
\ban
t^i:=t^{1i},\quad\phi_i:=\phi_{1i}.
\nan

From the Riemann-Hurwitz formula, we have
\ban
H^0_{0,d}(\mu)=\left\{\begin{array}{cc}
1,&\textrm{if }d=1\textrm{ and }\mu=(1),\\
0,&\textrm{otherwise}.
\end{array}\right.
\nan
So from \eqref{potential}, \eqref{potentialA}, \eqref{potentialB}, the genus-zero potential of $\bP^1_a$ has the following simple form:
\ban
F=\frac 12(t^{00})^2t^{01}+\frac{1}{2a}\suml_{i=1}^{a-1}t^{00}t^it^{a-i}+A^a+B^a_1Qe^{t^{01}},
\nan
with $A^a,B^a_1\in\bQ[t^1,\cdots,t^{a-1}]$. So the big quantum product is given by
\ban
\phi_i\star_\mathbf t\phi_j&=&\delta_{j,a-i}\frac{\phi_{01}}{a}+\suml_{k=1}^{a-1}A^a_{i,j,a-k}\cdot a\phi_{k}+
Qe^{t^{01}}\bigg((B^a_1)_{i,j}+\suml_{k=1}^{a-1}(B^a_1)_{i,j,a-k}\cdot a\phi_{k}\bigg),\\
\phi_i\star_\mathbf t\phi_{01}&=&Qe^{t^{01}}\bigg((B^a_1)_i+\suml_{k=1}^{a-1}(B^a_1)_{i,a-k}\cdot a\phi_{k}\bigg),\\
\phi_{01}\star_\mathbf t\phi_{01}&=&Qe^{t^{01}}\bigg(B^a_1+\suml_{k=1}^{a-1}(B^a_1)_{a-k}\cdot a\phi_{k}\bigg).
\nan

\begin{lemma}\label{B}
$B^a_1=t^1+O(t^{>1})$.
\end{lemma}
\begin{proof}
This follows from the dimension axiom and the fact that $\<\phi_1\>_{0,1,1}^{\bP^1_a}=1$.
\end{proof}

\begin{lemma}\label{Bzero}
For $i=1,\cdots,a-1$, we have $(B^a_1)_{i,a-i}=0$.
\end{lemma}
\begin{proof}
From \eqref{degt} and \eqref{degf}, we have
\ban
\deg t^i=\frac{a-i}{a},\quad\deg B^a_1=\frac{a-1}{a}.
\nan
Therefore, 
\ban
\deg(B^a_1)_{i,a-i}=\frac{a-1}{a}-\frac{a-i}{a}-\frac{i}{a}=\frac{-1}{a}<0.
\nan
\end{proof}

\begin{lemma}
The quantum Euler class has the form
\ban
e_q=(a+1)\phi_{01}+a^2\suml_{k=1}^{a-1}\suml_{i=1}^{a-1}A^a_{i,a-i,a-k}\phi_{k}.
\nan
\end{lemma}
\begin{proof}
From \eqref{qeuler}, we have
\ban
e_q=(a+1)\phi_{01}+a^2\suml_{k=1}^{a-1}\suml_{i=1}^{a-1}A^a_{i,a-i,a-k}\phi_{k}+Qe^{t^{01}}\bigg(a\suml_{i=1}^{a-1}(B^a_1)_{i,a-i}+a^2\suml_{k=1}^{a-1}\suml_{i=1}^{a-1}(B^a_1)_{i,a-i,a-k}\phi_{k}\bigg).
\nan
Now the required formula follows from Lemma \ref{Bzero}.
\end{proof}

We observe that $e_q\star_\mathbf t\phi_{01}=O(Qe^{t^{01}})$. Moreover, consider the coefficients of $e_q\star_\mathbf t\phi_{01}$ with respect to $\mathscr B$, we have the following observation.
\begin{lemma}
The coefficient of $e_q\star_\mathbf t\phi_{01}$ at $\phi_{00}$ is $2B^a_1Qe^{t^{01}}$.
\end{lemma}
\begin{proof}
One can check that
\ban
\<e_q\star_\mathbf t\phi_{01},\phi_{01}\>_{orb}^{\bP^1_a}&=&Qe^{t^{01}}\bigg((a+1)B^a_1+a^2\suml_{k=1}^{a-1}\suml_{i=1}^{a-1}A^a_{i,a-i,a-k}(B^a_1)_k\bigg).
\nan
Using WDVV \eqref{WDVV} of type $(i,a-i;(01),(01))$, we obtain
\ba\label{AB}
\suml_{k=1}^{a-1}A^a_{i,a-i,a-k}(B^a_1)_k=-\frac{B^a_1}{a^2},
\na
which implies the required result.
\end{proof}

Direct calculation gives
\ban
e_q\star_\mathbf t\phi_j=a\suml_{i=1}^{a-1}A^a_{i,a-i,j}\phi_{01}+a^3\suml_{l=1}^{a-1}\suml_{k=1}^{a-1}\suml_{i=1}^{a-1}A^a_{i,a-i,a-k}A^a_{k,j,a-l}\phi_l+O(Qe^{t^{01}}).
\nan
Now consider the matrix $M$ of $(e_q\star_\mathbf t)$ with respect to the basis $\mathscr B$:
\ban
(e_q\star_\mathbf t)[\phi_{00},\phi_{01},\phi_1,\cdots,\phi_{a-1}]=[\phi_{00},\phi_{01},\phi_1,\cdots,\phi_{a-1}]M.
\nan
Here $M$ has the form
\ban
\left[\begin{array}{ccccc}
0&2B^a_1Qe^{t^{01}}&O(Qe^{t^{01}})&\cdots&O(Qe^{t^{01}})\\
a+1&O(Qe^{t^{01}})&a\suml_{i=1}^{a-1}A^a_{i,a-i,1}+O(Qe^{t^{01}})&\cdots&a\suml_{i=1}^{a-1}A^a_{i,a-i,a-1}+O(Qe^{t^{01}})\\
a^2\suml_{i=1}^{a-1}A^a_{i,a-i,a-1}&O(Qe^{t^{01}})&a^3\suml_{k=1}^{a-1}\suml_{i=1}^{a-1}A^a_{i,a-i,a-k}A^a_{k,1,a-1}+O(Qe^{t^{01}})&\cdots&a^3\suml_{k=1}^{a-1}\suml_{i=1}^{a-1}A^a_{i,a-i,a-k}A^a_{k,a-1,a-1}+O(Qe^{t^{01}})\\
\vdots&\vdots&\vdots&&\vdots\\
a^2\suml_{i=1}^{a-1}A^a_{i,a-i,1}&O(Qe^{t^{01}})&a^3\suml_{k=1}^{a-1}\suml_{i=1}^{a-1}A^a_{i,a-i,a-k}A^a_{k,1,1}+O(Qe^{t^{01}})&\cdots&a^3\suml_{k=1}^{a-1}\suml_{i=1}^{a-1}A^a_{i,a-i,a-k}A^a_{k,a-1,1}+O(Qe^{t^{01}})
\end{array}\right].
\nan
Observe that the second column of $M$ is $O(Qe^{t^{01}})$. Recall that $\det$ is a multilinear function on column vectors. So we can take out the common factor $Qe^{t^{01}}$ in the second column to obtain
\ban
\det(e_q\star_\mathbf t)=\det M=Qe^{t^{01}}\det M_1+o(Qe^{t^{01}}),
\nan
where
\ban
M_1=\left[\begin{array}{ccccc}
0&2B^a_1&0&\cdots&0\\
a+1&*&a\suml_{i=1}^{a-1}A^a_{i,a-i,1}&\cdots&a\suml_{i=1}^{a-1}A^a_{i,a-i,a-1}\\
a^2\suml_{i=1}^{a-1}A^a_{i,a-i,a-1}&*&a^3\suml_{k=1}^{a-1}\suml_{i=1}^{a-1}A^a_{i,a-i,a-k}A^a_{k,1,a-1}&\cdots&a^3\suml_{k=1}^{a-1}\suml_{i=1}^{a-1}A^a_{i,a-i,a-k}A^a_{k,a-1,a-1}\\
\vdots&\vdots&\vdots&&\vdots\\
a^2\suml_{i=1}^{a-1}A^a_{i,a-i,1}&*&a^3\suml_{k=1}^{a-1}\suml_{i=1}^{a-1}A^a_{i,a-i,a-k}A^a_{k,1,1}&\cdots&a^3\suml_{k=1}^{a-1}\suml_{i=1}^{a-1}A^a_{i,a-i,a-k}A^a_{k,a-1,1}
\end{array}\right].
\nan
So 
\ba\label{det}
\det(e_q\star_\mathbf t)=-2B_1^aQe^{t^{01}}\det M_2+o(Qe^{t^{01}}),
\na
where
\ban
M_2=\left[\begin{array}{cccc}
a+1&a\suml_{i=1}^{a-1}A^a_{i,a-i,1}&\cdots&a\suml_{i=1}^{a-1}A^a_{i,a-i,a-1}\\
a^2\suml_{i=1}^{a-1}A^a_{i,a-i,a-1}&a^3\suml_{k=1}^{a-1}\suml_{i=1}^{a-1}A^a_{i,a-i,a-k}A^a_{k,1,a-1}&\cdots&a^3\suml_{k=1}^{a-1}\suml_{i=1}^{a-1}A^a_{i,a-i,a-k}A^a_{k,a-1,a-1}\\
\vdots&\vdots&&\vdots\\
a^2\suml_{i=1}^{a-1}A^a_{i,a-i,1}&a^3\suml_{k=1}^{a-1}\suml_{i=1}^{a-1}A^a_{i,a-i,a-k}A^a_{k,1,1}&\cdots&a^3\suml_{k=1}^{a-1}\suml_{i=1}^{a-1}A^a_{i,a-i,a-k}A^a_{k,a-1,1}
\end{array}\right].
\nan

\begin{lemma}\label{Aiik}
\ban
\suml_{i=1}^{a-1}A^a_{i,a-i,a-k}=\left\{\begin{array}{cc}
-\frac{a-1}{a^2}t^1+O(t^{>1}),&k=1,\\\\
O(t^{>1}),&2\leq k\leq a-1.
\end{array}\right.
\nan
\end{lemma}
\begin{proof}
From \eqref{degt} and \eqref{degf}, we have 
\ban
\deg A_{i,a-i,a-k}=2-\frac{a-i}{a}-\frac{i}{a}-\frac{k}{a}=\frac{a-k}{a}.
\nan
So for $k=2,\dots,a-1$, the polynomial $A_{i,a-i,a-k}$ lives in the ideal generated by $t^2,\cdots,t^{a-1}$. Moreover, from Lemma \ref{B} and formula \eqref{AB}, we have
\ban
\suml_{i=1}^{a-1}\suml_{k=1}^{a-1}A^a_{i,a-i,a-k}(B^a_1)_k=-\frac{a-1}{a^2}B^a_1=-\frac{a-1}{a^2}t^1+O(t^{>1}).
\nan
Since $A_{i,a-i,a-k}\in O(t^{>1})$ for $k=2,\cdots,a-1$, it follows that
\ban
\suml_{i=1}^{a-1}A^a_{i,a-i,a-1}(B^a_1)_1=-\frac{a-1}{a^2}t^1+O(t^{>1}).
\nan
Now the required equality comes from $(B^a_1)_1=1+O(t^{>1})$, which is a result of Lemma \ref{B}.
\end{proof}

\begin{lemma}\label{A1jl}
For $j,l=1,\cdots,a-1$, we have
\ban
A^a_{1,j,a-l}=\left\{\begin{array}{cc}
\frac 1a+O(t^{>1}),&l=j+1,\\\\
-\frac{t^1}{a^2}+O(t^{>1}),&(j,l)=(a-1,1),\\\\
O(t^{>1}),&else.
\end{array}\right.
\nan
\end{lemma}
\begin{proof}
From \eqref{degt} and \eqref{degf}, we have
\ban
\deg A^a_{1,j,a-l}=2-\frac{a-1}a-\frac{a-j}a-\frac la=\frac{j+1-l}a.
\nan
So
\ban
A^a_{1,j,a-l}\neq O(t^{>1})\Rightarrow(a-1)|(j+1-l)\Rightarrow j+1-l=0\textrm{ or }a-1.
\nan
If $j+1-l=0$, then
\ban
A^a_{1,j,a-j-1}|_{t^{>1}=0}=\<\phi_1,\phi_j,\phi_{a-j-1}\>_{0,3,0}^{\bP^1_a}=\frac 1a.
\nan 
If $j+1-l=a-1$, then $(j,a-l)=(a-1,a-1)$, and the required result follows from Corollary 3.34 and 3.35 in \cite{IST15}.
\end{proof}

From Lemma \ref{Aiik} and \ref{A1jl}, we have
\ba\label{M2}
M_2=\left[\begin{array}{ccccc}
a+1&0&\cdots&0&-\frac{a-1}at^1\\
-(a-1)t^1&0&\cdots&0&\frac{a-1}a(t^1)^2\\
0&-(a-1)t^1&\cdots&0&0\\
\vdots&\vdots&\ddots&\vdots&\vdots\\
0&0&\cdots&-(a-1)t^1&0
\end{array}\right]+O(t^{>1}).
\na
Now from Lemma \ref{B} and formula \eqref{det}, \eqref{M2}, we have
\ban
\det(e_q\star_\mathbf t)=Qe^{t^{01}}\big[-4\frac{(a-1)^{a-1}}a(t^1)^{a+1}+O(t^{>1})\big]+o(Qe^{t^{01}})\neq0.
\nan
This shows the semisimplicity of $QH(\bP^1_a)$.

\subsection{General case}

In this subsection, we prove the semisimplicity of $QH(\bP^1_\mathbf a)$, where $\mathbf a=(a_1,\cdots,a_r)$ is a tuple of postive integers with $r\geq1$ and each $a_\alpha\geq2$. Some results in the last two subsections will be used.

From \eqref{potential} and Corollary \ref{decompositionAB}, the genus-zero potential of $\bP^1_\mathbf a$ has the form:
\ban
F=\frac 12(t^{00})^2t^{01}+\suml_{\alpha=1}^r\suml_{i=1}^{a_\alpha-1}\frac{t^{00}t^{\alpha,i}t^{\alpha,{a_\alpha-i}}}{2a_\alpha}+\suml_{\alpha=1}^rA^{a_\alpha}(t^{\alpha,1},\cdots,t^{\alpha,a_\alpha-1})+B^\mathbf a_1Qe^{t^{01}}+o(Qe^{t^{01}}),
\nan
with
\ban
B^\mathbf a_1(\mathbf t^\mathbf a)=\prodl_{\alpha=1}^rB^{a_\alpha}_1(t^{\alpha,1},\cdots,t^{\alpha,a_\alpha-1}).
\nan
In particular, from Lemma \ref{B}, we have
\ba
B^\mathbf a_1(\mathbf t^\mathbf a)=\prodl_{\alpha=1}^rt^{\alpha1}+O(t^{1,>1},\cdots,t^{r,>1}).
\na
So the big quantum product is given by
\ban
\phi_{\alpha i}\star_\mathbf t\phi_{\alpha j}&=&\delta_{j,a_\alpha-i}\frac{\phi_{01}}{a_\alpha}+\suml_{k=1}^{a_\alpha-1}A^{a_\alpha}_{i,j,a_\alpha-k}\cdot a_\alpha\phi_{\alpha k}+
Qe^{t^{01}}B^\mathbf a_1\bigg(\frac{(B^{a_\alpha}_1)_{i,j}}{B^{a_\alpha}_1}\\
&&\quad+\suml_{k=1}^{a_\alpha-1}\frac{(B^{a_\alpha}_1)_{i,j,a_\alpha-k}}{B^{a_\alpha}_1}a_\alpha\phi_{\alpha k}+\frac{(B^{a_\alpha}_1)_{i,j}}{B^{a_\alpha}_1}\suml_{\beta\neq\alpha}\suml_{k=1}^{a_\beta-1}\frac{(B^{a_\beta}_1)_{a_\beta-k}}{B^{a_\beta}_1}a_\beta\phi_{\beta k}\bigg)+o(Qe^{t^{01}}),\\
\phi_{\alpha i}\star_\mathbf t\phi_{\beta j}&=&Qe^{t^{01}}B^\mathbf a_1\bigg(\frac{(B^{a_\alpha}_1)_i(B^{a_\beta}_1)_j}{B^{a_\alpha}_1B^{a_\beta}_1}+\suml_{k=1}^{a_\alpha-1}\frac{(B^{a_\alpha}_1)_{i,a_\alpha-k}(B^{a_\beta}_1)_j}{B^{a_\alpha}_1B^{a_\beta}_1}a_\alpha\phi_{\alpha k}+\suml_{k=1}^{a_\beta-1}\frac{(B^{a_\alpha}_1)_i(B^{a_\beta}_1)_{j,a_\beta-k}}{B^{a_\alpha}_1B^{a_\beta}_1}a_\beta\phi_{\beta k}\\
&&\quad+\frac{(B^{a_\alpha}_1)_i(B^{a_\beta}_1)_j}{B^{a_\alpha}_1B^{a_\beta}_1}\suml_{\gamma\neq\alpha,\beta}\suml_{k=1}^{a_\gamma-1}\frac{(B^{a_\gamma}_1)_{a_\gamma-k}}{B^{a_\gamma}_1}a_\gamma\phi_{\gamma k}\bigg)+o(Qe^{t^{01}}),\quad\alpha\neq\beta,\\
\phi_{\alpha i}\star_\mathbf t\phi_{01}&=&Qe^{t^{01}}B^\mathbf a_1\bigg(\frac{(B^{a_\alpha}_1)_{i}}{B^{a_\alpha}_1}+\suml_{k=1}^{a_\alpha-1}\frac{(B^{a_\alpha}_1)_{i,a_\alpha-k}}{B^{a_\alpha}_1}a_\alpha\phi_{\alpha k}+\frac{(B^{a_\alpha}_1)_{i}}{B^{a_\alpha}_1}\suml_{\beta\neq\alpha}\suml_{k=1}^{a_\beta-1}\frac{(B^{a_\beta}_1)_{a_\beta-k}}{B^{a_\beta}_1}a_\beta\phi_{\beta k}\bigg)+o(Qe^{t^{01}}),\\
\phi_{01}\star_\mathbf t\phi_{01}&=&Qe^{t^{01}}B^\mathbf a_1\bigg(1+\suml_{\alpha=1}^r\suml_{k=1}^{a_\alpha-1}\frac{(B^{a_\alpha}_1)_{a_\alpha-k}}{B^{a_\alpha}_1}a_\alpha\phi_{\alpha k}\bigg)+o(Qe^{t^{01}}).
\nan

From \eqref{qeuler}, we can use Lemma \ref{Bzero} to check that the quantum Euler class has the form
\ban
e_q=\big(2+\suml_{\alpha=1}^r(a_\alpha-1)\big)\phi_{01}+\suml_{\alpha=1}^ra_\alpha^2\suml_{k=1}^{a_\alpha-1}\suml_{i=1}^{a_\alpha-1}A^{a_\alpha}_{i,a_\alpha-i,a_\alpha-k}\phi_{\alpha k}+o(Qe^{t^{01}}).
\nan
We observe that $e_q\star_\mathbf t\phi_{01}=O(Qe^{t^{01}})$. Moreover, consider the coefficients of $e_q\star_\mathbf t\phi_{01}$ with respect to $\mathscr B$, we have the following observation.
\begin{lemma}
The coefficient of $e_q\star_\mathbf t\phi_{01}$ at $\phi_{00}$ is $2B^\mathbf a_1Qe^{t^{01}}+o(Qe^{t^{01}})$.
\end{lemma}
\begin{proof}
One can check that
\ban
\<e_q\star_\mathbf t\phi_{01},\phi_{01}\>_{orb}^{\bP^1_\mathbf a}&=&B^\mathbf a_1Qe^{t^{01}}\bigg(2+\suml_{\alpha=1}^r(a_\alpha-1)+\suml_{\alpha=1}^ra_\alpha^2\suml_{k=1}^{a_\alpha-1}\suml_{i=1}^{a_\alpha-1}A^{a_\alpha}_{i,a_\alpha-i,a_\alpha-k}\frac{(B^{a_\alpha}_1)_k}{B^{a_\alpha}_1}\bigg)+o(Qe^{t^{01}}).
\nan
Now we can use \eqref{AB} to conclude the required result.
\end{proof}

Direct calculation gives
\ban
e_q\star_\mathbf t\phi_{\alpha j}=a_\alpha\suml_{i=1}^{a_\alpha-1}A^{a_\alpha}_{i,a_\alpha-i,j}\phi_{01}+a_\alpha^3\suml_{l=1}^{a_\alpha-1}\suml_{k=1}^{a_\alpha-1}\suml_{i=1}^{a_\alpha-1}A^{a_\alpha}_{i,a_\alpha-i,a_\alpha-k}A^{a_\alpha}_{k,j,a_\alpha-l}\phi_{\alpha l}+O(Qe^{t^{01}}).
\nan

For $\alpha=1,\cdots,r$, let $\vec\phi_{\alpha}:=[\phi_{\alpha 1},\cdots,\phi_{\alpha,a_\alpha-1}]$ be a row vector. Consider the matrix $M$ of $(e_q\star_\mathbf t)$ with respect to the basis $\mathscr B$:
\ban
(e_q\star_\mathbf t)[\phi_{00},\phi_{01},\vec\phi_{1},\cdots,\vec\phi_{r}]=[\phi_{00},\phi_{01},\vec\phi_{1},\cdots,\vec\phi_{r}]M.
\nan
Then $M$ is $\bQ[t^1,\cdots,t^{a-1}][[Qe^{t^{01}}]]$-valued, and the second column of $M$ is $O(Qe^{t^{01}})$. Recall that $\det$ is a multilinear function on column vectors. So we can take out the common factor $Qe^{t^{01}}$ in the second column of $M$ to obtain
\ban
\det(e_q\star_\mathbf t)=\det M=Qe^{t^{01}}\det M_1+o(Qe^{t^{01}}),
\nan
with
\ban
M_1=\left[\begin{array}{ccccc}
0&2B^\mathbf a_1&0&\cdots&0\\
2+\suml_{\alpha=1}^r(a_\alpha-1)&*&\vec r_1&\cdots&\vec r_r\\
\vec c_1&*&b_1&\cdots&0\\
\vdots&\vdots&\vdots&\ddots&\vdots\\
\vec c_r&*&0&\cdots&b_r
\end{array}\right].
\nan
Here for $\alpha=1,\cdots,r$,
\ban
\vec r_\alpha&=&[a_\alpha\suml_{i=1}^{a_\alpha-1}A^{a_\alpha}_{i,a_\alpha-i,1},\cdots,a_\alpha\suml_{i=1}^{a_\alpha-1}A^{a_\alpha}_{i,a_\alpha-i,a_\alpha-1}],\\
\vec c_\alpha&=&[a_\alpha\suml_{i=1}^{a_\alpha-1}A^{a_\alpha}_{i,a_\alpha-i,a_\alpha-1},\cdots,a_\alpha\suml_{i=1}^{a_\alpha-1}A^{a_\alpha}_{i,a_\alpha-i,1}]^T,\\
b_\alpha&=&\left[\begin{array}{ccc}
a_\alpha^3\suml_{l=1}^{a_\alpha-1}\suml_{k=1}^{a_\alpha-1}\suml_{i=1}^{a_\alpha-1}A^{a_\alpha}_{i,a_\alpha-i,a_\alpha-k}A^{a_\alpha}_{k,1,a_\alpha-1}&\cdots&a_\alpha^3\suml_{l=1}^{a_\alpha-1}\suml_{k=1}^{a_\alpha-1}\suml_{i=1}^{a_\alpha-1}A^{a_\alpha}_{i,a_\alpha-i,a_\alpha-k}A^{a_\alpha}_{k,a_\alpha-1,a_\alpha-1}\\
\vdots&&\vdots\\
a_\alpha^3\suml_{l=1}^{a_\alpha-1}\suml_{k=1}^{a_\alpha-1}\suml_{i=1}^{a_\alpha-1}A^{a_\alpha}_{i,a_\alpha-i,a_\alpha-k}A^{a_\alpha}_{k,1,1}&\cdots&a_\alpha^3\suml_{l=1}^{a_\alpha-1}\suml_{k=1}^{a_\alpha-1}\suml_{i=1}^{a_\alpha-1}A^{a_\alpha}_{i,a_\alpha-i,a_\alpha-k}A^{a_\alpha}_{k,a_\alpha-1,1}
\end{array}\right].
\nan
So we have
\ban
\det(e_q\star_\mathbf t)=-2B^\mathbf a_1Qe^{t^{01}}\det M_2+o(Qe^{t^{01}}),
\nan
with
\ban
M_2=\left[\begin{array}{cccc}
2+\suml_{\alpha=1}^r(a_\alpha-1)&\vec r_1&\cdots&\vec r_r\\
\vec c_1&b_1&\cdots&0\\
\vdots&\vdots&\ddots&\vdots\\
\vec c_r&0&\cdots&b_r
\end{array}\right].
\nan
From Lemma \ref{Aiik} and \ref{A1jl}, we have
\ban
\vec r_\alpha&=&[0,\cdots,0,-\frac{a_\alpha-1}{a_\alpha}t^{\alpha1}]+O(t^{\alpha,>1}),\\
\vec c_\alpha&=&[-(a_\alpha-1)t^{\alpha1},0,\cdots,0]^T+O(t^{\alpha,>1}),\\
b_\alpha&=&\left[\begin{array}{cccc}
0&\cdots&0&\frac{a_\alpha-1}{a_\alpha}(t^{\alpha1})^2\\
-(a_\alpha-1)t^{\alpha1}&\cdots&0&0\\
\vdots&\ddots&\vdots&\vdots\\
0&\cdots&-(a_\alpha-1)t^{\alpha1}&0
\end{array}\right]+O(t^{\alpha,>1}).
\nan
As a consequence,
\ban
\det(e_q\star_\mathbf t)=Qe^{t^{01}}\big[-4\prodl_{\alpha=1}^r\frac{(a_\alpha-1)^{a_\alpha-1}}{a_\alpha}(t^{\alpha1})^{a_\alpha+1}+O(t^{1,>1},\cdots,t^{r,>1})\big]+o(Qe^{t^{01}})\neq0.
\nan
This shows the semisimplicity of $QH(\bP^1_\mathbf a)$.

\subsection{Proof of Corollary \ref{cor}}

In this subsection, we prove Corollary \ref{cor}.

If $\sC$ is a $\bP^1$-orbifold, then the semisimplicity of $QH(\sC)$ comes from Proposition \ref{big}, and the existence of a full exceptional collection in $D^b(\sC)$ was proved by Geigle-Lenzing (Proposition 4.1 in \cite{GL}). 

Now assume that $\sC$ is not a $\bP^1$-orbifold. On one hand, since the genus of $|\sC|$ is positive, it follows that the quantum product of $\sC$ is identical to the Chen-Ruan product, which implies that $QH(\sC)$ contains nilpotent elements, and hence not semisimple. On the other hand, since the odd cohomology group of $|\sC|$ does not vanish, it follows that we can use the orbifold HKR isomorphism (see Section 1.15-1.17 in \cite{ACH}) to conclude that $D^b(\sC)$ does not admit a full exceptional collection.

This finishes the proof of Corollary \ref{cor}.

\section{Small quantum cohomology}

In this section, we prove Proposition \ref{small}. To show the ``only if" part, from Corollary \ref{nilpotent}, it suffices to prove the following lemma.
\begin{lemma}
If $\chi_{orb}(\sC)\leq 0$, then $qH(\sC)$ contains nilpotent elements.
\end{lemma}
\begin{proof}
For homogenous $x,y\in H^{even}_{orb}(\sC)$, we have
\ban
x\circ y=x\cup_{CR}y+\suml_{s\in\mathscr S}\suml_{d>0}\<x,y,\phi_s\>_{0,3,d}^{\sC}Q^d\phi^s.
\nan
Note that $\deg_{orb}(x\cup_{CR}y)=\deg_{orb}x+\deg_{orb}y$. If for some $d>0$ and $s\in\mathscr S$ we have
\ban
\<x,y,\phi_s\>_{0,3,d}^{\sC}\neq0,
\nan
then the dimension constraint gives
\ban
&&\deg_{orb}x+\deg_{orb}y+\deg_{orb}\phi_s=2+2d\cdot\chi_{orb}(\sC)\leq2\\
&\Rightarrow&\deg_{orb}\phi^s\geq\deg_{orb}x+\deg_{orb}y.
\nan
As a consequence, $x\circ y$ is a linear combination of $\{\phi^s:s\in\mathscr S,\textrm{ and }\deg_{orb}\phi^s\geq\deg_{orb}x+\deg_{orb}y\}$ with coefficients in $\bQ[[Q]]$. In particular, if $\deg_{orb}x>0$, then $x$ is nilpotent.
\end{proof}

To prove the ``if" part of Proposition \ref{small}, we first use the dimension axiom to observe that for a Fano orbi-curve $\sC$, the coefficients in the small quantum product take values in $\bC[Q]$. So it suffices to show the generic semisimplicity of 
\ban
\bar qH(\sC):=H^{even}_{orb}(C_\mathbf a)\otimes_\bC\bC[Q]
\nan
over $\bC[Q]$, since we have the follwing Cartesian diagram
\[
\begin{CD}
\tspec\big(qH(\sC)\big)@>>>\tspec\big(\bar qH(\sC)\big)\\
@VVV @VVV \\
\tspec\bC[[Q]]@>>> \tspec\bC[Q],
\end{CD}
\]
where the base morphism is dominant, given by the natural injective map $\bC[Q]\rightarrow\bC[[Q]]$.

Recall that, from \eqref{CR1} and \eqref{CR2}, the orbifold cohomogy ring $H^{even}_{orb}(C_\mathbf a)$ is generated by $\phi_{\alpha1}$'s over $\bC$, with relations
\ban
f_{\alpha\beta}=\phi_{\alpha1}\cup_{CR}\phi_{\beta1}=0,\quad g_{\alpha\beta}=a_\alpha\underbrace{\phi_{\alpha1}\cup_{CR}\cdots\cup_{CR}\phi_{\alpha1}}_{a_\alpha}-a_\beta\underbrace{\phi_{\beta1}\cup_{CR}\cdots\cup_{CR}\phi_{\beta1}}_{a_\beta}=0,\quad\alpha\neq\beta.
\nan
So $\bar qH(\sC)$ is generated by $\phi_{\alpha1}$'s over $\bC[Q]$, with new relations $f'_{\alpha\beta},g'_{\alpha\beta}(\alpha\neq\beta)$. Here the new relation $f'_{\alpha\beta}$ (resp. $g'_{\alpha\beta}$) is just the relation $f_{\alpha\beta}$ (resp. $g_{\alpha\beta}$) evaluated in the small quantum cohomology ring structure. In other words, we have the following presentation for $\bar qH(\sC)$:
\ban
\bar qH(\sC)\cong\bC[Q][x_1,\cdots,x_{r}]/I,\textrm{ with }\phi_{\alpha1}\mapsto x_\alpha,
\nan
where the ideal $I$ is generated by the relations $f'_{\alpha\beta},g'_{\alpha\beta}(\alpha\neq\beta)$. Note that at $Q=1$, $\bar qH(\sC)_{Q=1}$ is a $\bC$-algebra of dimension $N=2+\suml_{\alpha=1}^r(a_\alpha-1)$. Our strategy is to show that the ideal $I_{Q=1}$ determines exactly $N$ distinct points in $\bC^r$, from which the semisimplicity of $\bar qH(\sC)_{Q=1}$ over $\bC$ follows.  From Lemma \ref{ABsemisimple}, this implies the generic semisimplicity of $\bar qH(\sC)$ over $\bC[Q]$.

Now we check the ``if" part of Proposition \ref{small} case by case.

\underline{$\sC=\bP^1_{a_1,a_2}(a_1,a_2\geq1)$:} We have the following presentation for $\bar qH(\bP^1_{a_1,a_2})$ (see (4.32) in \cite{MT}): 
\ban
\bar qH(\bP^1_{a_1,a_2})\cong\bC[Q][x_1,x_2]/I,\textrm{ with }\phi_{11}\mapsto x_1,\phi_{21}\mapsto x_2,
\nan
where the ideal $I$ is generated by $ x_1x_2-Q$ and $a_1x_1^{a_1}-a_2x_2^{a_2}$. Direct calculation gives the set of solutions of $I_{Q=1}$:
\ban
x_1=(\frac{a_2}{a_1})^{\frac{1}{a_1+a_2}}\xi_{a_1+a_2}^k,\quad x_2=(\frac{a_1}{a_2})^{\frac{1}{a_1+a_2}}\xi_{a_1+a_2}^{-k}\quad(1\leq k\leq a_1+a_2).
\nan
Here for a positive integer $N$, $\xi_N:=e^{\frac{2\pi\sqrt{-1}}{N}}$.
\begin{remark}
The semisimplicity of $qH(\bP^1_{a_1,a_2})$ was also pointed out by Milanov-Tseng (see Section 4.4 in \cite{MT}).
\end{remark}

\underline{$\sC=\bP^1_{2,2,a}(a\geq2)$:}
We have the following presentation for $\bar qH(\bP^1_{2,2,a})$ (see Section 5 in \cite{R}):
\ban
\bar qH(\bP^1_{2,2,a})\cong\bC[Q][x,y,z]/I,\textrm{ with }\phi_{11}\mapsto x,\phi_{21}\mapsto y,\phi_{31}\mapsto z,
\nan
where the ideal $I$ is generated by
\ban
&&xy+a\suml_{k=0}^{\lfloor\frac{a-1}{2}\rfloor}(-1)^{k-1}{a-1-k\choose k}Q^{2k+1}z^{a-1-2k},\\
&\textrm{and }&xz-2Qy,\quad yz-2Qx.
\nan
Using formulae (2.3) and (2.4) in \cite{GQ}, we can solve the equations directly. For $a=2m$, the set of solutions of $I_{Q=1}$ is
\ban
(\pm a,\pm a,2),\quad(\pm a,\mp a,-2),\quad(0,0,0),\quad(0,0,\pm\sqrt{2-2\cos\frac{k\pi}{m}})\quad(1\leq k\leq m-1),
\nan
and for $a=2m+1$, the set of solutions of $I_{Q=1}$ is
\ban
(\pm a,\pm a,2),\quad(\pm a\sqrt{-1},\mp a\sqrt{-1},-2),\quad(0,0,\pm\sqrt{2-2\cos\frac{(2k+1)\pi}{2m+1}})\quad(0\leq k\leq m-1).
\nan

\underline{$\sC=\bP^1_{2,3,3}$:}
Using the explicit formula for the genus-zero potential (see Appendix A.1 in \cite{IST12}), we have the following presentation for $\bar qH(\bP^1_{2,3,3})$:
\ban
\bar qH(\bP^1_{2,3,3})\cong\bC[Q][x,y,z]/I,\textrm{ with }\phi_{11}\mapsto x,\phi_{21}\mapsto y,\phi_{31}\mapsto z,
\nan
where the ideal $I$ is generated by
\ban
xy-3Qz^2+6Q^3y,\quad xz-3Qy^2+6Q^3z,\quad yz-2Qx-4Q^4.
\nan
Direct calculation gives the set of solutions of $I_{Q=1}$:
\ban
(-2,0,0),\quad(0,2\xi_3^k,2\xi_3^{2k}),\quad(6,4\xi_3^k,4\xi_3^{2k})\quad(k=0,1,2).
\nan

\underline{$\sC=\bP^1_{2,3,4}$:}
Using the explicit formula for the genus-zero potential (see Appendix A.5 in \cite{IST12}), we have the following presentation for $\bar qH(\bP^1_{2,3,4})$:
\ban
\bar qH(\bP^1_{2,3,4})\cong\bC[Q][x,y,z]/I,\textrm{ with }\phi_{11}\mapsto x,\phi_{21}\mapsto y,\phi_{31}\mapsto z,
\nan
where the ideal $I$ is generated by
\ban
&&xy-4Qz^3+28Q^4x+72Q^7z,\quad xz-3Qy^2+8Q^3z^2-18Q^5y-24Q^9,\\
&&yz-2Qx-4Q^4z.
\nan
We use MAPLE to get the set of solutions of $I_{Q=1}$:
\ban
(0,-4,0),\quad(0,-2,0),\quad(\pm4,0,\mp2),\quad(0,4,3\pm\sqrt2),\quad(\pm12,8,\pm6).
\nan

\underline{$\sC=\bP^1_{2,3,5}$:}
Using the explicit formula for the genus-zero potential (see Appendix A.9 in \cite{IST12}), we have the following presentation for $\bar qH(\bP^1_{2,3,5})$:
\ban
\bar qH(\bP^1_{2,3,5})\cong\bC[Q][x,y,z]/I,\textrm{ with }\phi_{11}\mapsto x,\phi_{21}\mapsto y,\phi_{31}\mapsto z,
\nan
where $I$ is the ideal generated by
\ban
&&xy=5Qz^4-129Q^5y^2+350Q^7z^3-2920Q^{10}x-8140Q^{13}z^2+14130Q^{15}y+20400Q^{19}z+76080Q^{25},\\
&&xz=3Qy^2-10Q^3z^3+72Q^6x+205Q^9z^2-360Q^{11}y-510Q^{15}z-1920Q^{21},\\
&&yz=2Qx+5Q^4z^2-12Q^6y-20Q^{10}z-60Q^{16}.
\nan
We use MAPLE to get the set of solutions of $I_{Q=1}$:
\ban
(0,0,-2),(6,-4,0),(30,20,12),(0,\pm5,\mp3),(0,10,3\pm3\sqrt5),(-10,0,2\pm2\sqrt5).
\nan

{\bf Acknowledgements}

The author is grateful to Yongbin Ruan and Jian Zhou for encouragement, and Jianxun Hu for constant support. The author would also like to thank Yuri Manin and Maxim Smirnov for their interest in this work, and Chengyong Du, Weiqiang He, Xiaowen Hu and Di Yang for helpful discussions.  This work is supported by grants of the National Natural Science Foundation of China [11601534, 11831017, 11771461, 11521101] and the Fundamental Research Funds for the Central Universities [SYSU18lgpy67].

\end{document}